\documentclass[12pt,reqno]{amsart}
\usepackage{amssymb}

\large\normalsize
\theoremstyle{plain}
\newtheorem{theorem}{Theorem}
\newtheorem{corollary}{Corollary}

\theoremstyle{definition}
\newtheorem{definition}{Definition}
\newtheorem{example}{Example}
\newtheorem{remark}{Remark}

\numberwithin{equation}{section}
\numberwithin{theorem}{section}
\numberwithin{corollary}{section}
\numberwithin{definition}{section}
\numberwithin{remark}{section}
\numberwithin{lemma}{section}

\def\R{{\mathbb R}}

\begin{document}

\title[fractional Sturm-Liouville bvp] {fractional-order boundary value problem with Sturm-Liouville boundary conditions}
\author[Anderson]{Douglas R. Anderson} 
\address{Department of Mathematics \\
         Concordia College \\
         Moorhead, MN 56562 USA}
\email{andersod@cord.edu}

\author[Avery]{Richard I. Avery} 
\address{College of Arts and Sciences \\ 
         Dakota State University\\
         Madison, SD 57042, USA}
\email{rich.avery@dsu.edu}



\begin{abstract}
Using the new conformable fractional derivative, which differs from the Riemann-Liouville and Caputo fractional derivatives, we reformulate the second-order conjugate boundary value problem in this new setting. Utilizing the corresponding positive fractional Green's function, we apply a functional compression-expansion fixed point theorem to prove the existence of a positive solution. We then compare our results favorably to those based on the Riemann-Liouville fractional derivative.
\end{abstract}

\maketitle

\section{Introduction}

The search for the existence of positive solutions and multiple positive solutions to nonlinear fractional boundary value problems has expanded greatly over the past decade; for some recent examples please see [3-9,11,14,15,17-20]. In all of these works and the references cited therein, however, the definition of the fractional derivative used is either the Caputo or the Riemann-Liouville fractional derivative, involving an integral expression and the gamma function. Recently \cite{udita,hammad,khalil} a new definition has been formulated and dubbed the conformable fractional derivative. In this paper, we use this fractional derivative of order $\alpha$, given by
\begin{equation}\label{derivdef}
 D^{\alpha}f(t):=\lim_{\varepsilon\rightarrow 0}\frac{f(te^{\varepsilon t^{-\alpha}})-f(t)}{\varepsilon}, \quad D^{\alpha}f(0)=\lim_{t\rightarrow 0^+}D^{\alpha}f(t);  
\end{equation}
note that if $f$ is differentiable, then
\begin{equation}\label{fracshort}
 D^{\alpha}f(t) = t^{1-\alpha} f'(t), 
\end{equation}
where $f'(t)=\lim_{\varepsilon\rightarrow 0}[f(t+\varepsilon)-f(t)]/\varepsilon$.
Using this new definition of the fractional derivative, we investigate a conformable fractional boundary value problem with Sturm-Liouville boundary conditions. With the fractional differential equation and fractional boundary conditions established, we find the corresponding Green's function and prove its positivity under appropriate assumptions. This work thus sets the stage for employing a functional compression-expansion fixed point theorem to prove the existence of a positive solution to the special case of conjugate boundary conditions. We then compare our existence result to that of Bai and L\"{u} \cite{bai}.

\section{Two Iterated Fractional Derivatives}

We begin by considering two iterated fractional derivatives in the differential operator, together with two-point boundary conditions, as illustrated in the nonlinear boundary value problem 
	\begin{gather}
	-D^{\beta}D^{\alpha} x(t) = f(t,x(t)), \qquad 0 \le t \le 1, \label{2eigeneqn} \\
	\gamma x(0)-\delta D^\alpha x(0) = 0 = \eta x(1)+\zeta D^{\alpha} x(1), \label{2eigenbc}
	\end{gather}
where	$\alpha,\beta\in(0,1]$ and the derivatives are conformable fractional derivatives \eqref{derivdef}, with $\gamma,\delta,\eta,\zeta\ge 0$ and $d:=\eta\delta+\gamma\zeta+\gamma\eta/\alpha>0$. Note that if $x$ is $\alpha$-differentiable and $t^{1-\alpha}x'$ is $\beta$-differentiable, then using \eqref{fracshort} we could rewrite \eqref{2eigeneqn} as
\[ 	-t^{1-\beta}\left(t^{1-\alpha} x'(t)\right)' = f(t,x(t)), \qquad 0 \le t \le 1, \]
where the prime indicates the derivative $d/dt$.
Before we find Green's function for \eqref{2eigeneqn}, \eqref{2eigenbc}, and prove that it is positive, we first define the $\beta$-fractional integral.


\begin{definition} 
Let $\beta\in(0,1]$ and $0 \le a < b$. A function $f:[a,b]\rightarrow\R$ is $\beta$-fractional integrable on $[a,b]$ if the integral
\[ \int_a^b f(s) d_{\beta}s:= \int_a^b f(s)s^{\beta-1}ds \]
exists and is finite. Thus the integral can be interpreted either as a Riemann-Stieltjes integral or an improper Riemann integral.
\end{definition}


\begin{theorem}
Let $\alpha,\beta\in(0,1]$. The corresponding Green's function for the homogeneous problem 
\[ -D^{\beta}D^{\alpha} x(t)=0 \] 
satisfying the boundary conditions \eqref{2eigenbc} is given by 
\begin{equation}\label{gf2}
	G(t,s)=
			\begin{cases}
				\frac{1}{d}\left[\delta+\frac{\gamma}{\alpha} s^\alpha\right]\left[\zeta+\frac{\eta}{\alpha}\left(1-t^\alpha\right)\right] &:s\le t, \\
				\frac{1}{d}\left[\delta+\frac{\gamma}{\alpha} t^\alpha\right]\left[\zeta+\frac{\eta}{\alpha}\left(1-s^\alpha\right)\right] &:t\le s, 
			\end{cases}
\end{equation}
where we assume the parameters satisfy $\gamma,\delta,\eta,\zeta\ge 0$ and $d=\eta\delta+\gamma\zeta+\gamma\eta/\alpha>0$.
\end{theorem}

\begin{proof}
We will show that
\[ x(t)=\int_{0}^{1}G(t,s)h(s)d_{\beta}s, \]
for $G$ given by \eqref{gf2}, is a solution to the linear boundary value problem
\[ -D^{\beta}D^{\alpha} x(t) = h(t) \]
with boundary conditions \eqref{2eigenbc}. 

For any $t\in[0,1]$, using the branches of \eqref{gf2} we have
\begin{eqnarray*} 
 x(t) &=& \frac{1}{d}\left[\zeta+\frac{\eta}{\alpha}\left(1-t^\alpha\right)\right] \int_{0}^{t} \left[\delta+\frac{\gamma}{\alpha} s^\alpha\right] h(s)d_{\beta}s \\
  & & + \frac{1}{d}\left[\delta+\frac{\gamma}{\alpha} t^\alpha\right] \int_{t}^{1} \left[\zeta+\frac{\eta}{\alpha}\left(1-s^\alpha\right)\right]h(s)d_{\beta}s.
\end{eqnarray*}
Taking the $\alpha$-fractional derivative yields
\begin{eqnarray*} 
 D^{\alpha}x(t) &=& -\frac{\eta}{d} \int_{0}^{t} \left[\delta+\frac{\gamma}{\alpha} s^\alpha\right] h(s)d_{\beta}s 
   + \frac{\gamma}{d} \int_{t}^{1} \left[\zeta+\frac{\eta}{\alpha}\left(1-s^\alpha\right)\right]h(s)d_{\beta}s.
\end{eqnarray*}
Checking the first boundary condition, we see that
\[ \gamma x(0)-\delta D^\alpha x(0) = 0. \]
Moreover, in checking the second boundary condition we get
\[ \eta x(1)+\zeta D^{\alpha} x(1) = 0. \]
Taking the $\beta$-fractional derivative of the $\alpha$-fractional derivative yields
\begin{eqnarray*} 
 D^{\beta}D^{\alpha}x(t) &=& -\frac{\eta}{d} \left[\delta+\frac{\gamma}{\alpha} t^\alpha\right] h(t)t^{\beta-1}t^{1-\beta} 
   - \frac{\gamma}{d} \left[\zeta+\frac{\eta}{\alpha}\left(1-t^\alpha\right)\right]h(t)t^{\beta-1}t^{1-\beta} \\
   &=& -\frac{1}{d}h(t)\left[ \eta\delta  + \gamma\zeta+\frac{\gamma\eta}{\alpha}\right] = -h(t),
\end{eqnarray*}
which is what we set out to prove.
\end{proof}


\begin{corollary}[Fractional Conjugate and Right-Focal Problems]
Let $\alpha,\beta\in(0,1]$. The corresponding Green's function for the homogeneous problem 
\[ -D^{\beta}D^{\alpha} x(t)=0 \] 
satisfying the conjugate boundary conditions $x(0)=x(1)=0$ is given by 
\begin{equation}\label{gfconj}
	G(t,s)=
			\begin{cases}
				\frac{1}{\alpha} s^\alpha\left(1-t^\alpha\right) &:s\le t, \\
				\frac{1}{\alpha} t^\alpha\left(1-s^\alpha\right) &:t\le s, 
			\end{cases}
\end{equation}
and the corresponding Green's function for the homogeneous problem 
\[ -D^{\beta}D^{\alpha} x(t)=0 \] 
satisfying the right-focal-type boundary conditions $x(0)=D^{\alpha}x(1)=0$ is given by 
\begin{equation}\label{gfrtfoc}
	G(t,s)=
			\begin{cases}
				\frac{1}{\alpha} s^\alpha &:s\le t, \\
				\frac{1}{\alpha} t^\alpha &:t\le s.
			\end{cases}
\end{equation}
\end{corollary}


\begin{remark}
Note that the conformable fractional Green's function based on \eqref{derivdef} and given above for the conjugate boundary conditions in \eqref{gfconj} differs from that found for example in Bai and L\"{u} \cite{bai}, where the Riemann-Liouville fractional derivative is used.
\end{remark}


\begin{theorem}[Bounds on Green's Function]
For $G(t,s)$ given in \eqref{gf2}, we have the following bounds. First,
\begin{equation}\label{Gbds2} 
  g_1(t)G(s,s) < G(t,s) \le G(s,s)
\end{equation}
for $t,s\in[0,1]$, where
\begin{equation}\label{g2}
  g_1(t):=\min\left\{\frac{\alpha\delta+\gamma  t^\alpha}{\alpha\delta+\gamma},\frac{\alpha\zeta+\eta\left(1-t^{\alpha}\right)}{\alpha\zeta+\eta}\right\}.
\end{equation}
Next, for integers $n\ge 3$,
\begin{equation}\label{Gboundsgen}
 \min_{\frac{1}{n}\le t \le 1-\frac{1}{n}}G(t,s) \ge g_2(s)G(s,s)
\end{equation}
for $g_2$ given by
\begin{equation}\label{ggen}
 g_2(s):= \begin{cases}
     \displaystyle\frac{\alpha\zeta+\eta\left(1-\left(1-1/n\right)^\alpha\right)}{\alpha\zeta+\eta\left(1-s^\alpha\right)} &: s\in[0,r] \\
     \displaystyle\frac{\alpha\delta+\gamma(1/n)^\alpha}{\alpha\delta+\gamma  s^\alpha} &: s\in[r,1]
  \end{cases}
\end{equation}
for the constant
\begin{equation}\label{rgendef} r=\left(\frac{\alpha\gamma\zeta+\gamma\eta+\alpha\delta\eta(n-1)^{\alpha}}{\left(\alpha\delta\eta+\alpha\gamma\zeta+\gamma\eta\right)n^{\alpha}-\gamma\eta(n-1)^{\alpha}+\gamma\eta}\right)^{1/\alpha}\in\left(\frac{1}{n},1-\frac{1}{n}\right],
\end{equation}
where $r=1-1/n$ if $\gamma=0$. Finally, we also have
\begin{equation}\label{Gbdconst}
 \min_{\frac{1}{n}\le t \le 1-\frac{1}{n}}G(t,s) \ge g_3G(s,s)
\end{equation}
for the constant $g_3$ given by
\begin{equation}\label{g3def}
 g_3\equiv \min\left\{\frac{\alpha\zeta+\eta\left(1-\left(1-1/n\right)^\alpha\right)}{\alpha\zeta+\eta}, \frac{\alpha\delta+\gamma(1/n)^\alpha}{\alpha\delta+\gamma}\right\} 
\end{equation}
for all $\alpha\in(0,1]$.
\end{theorem}

\begin{proof}
It is straightforward to see that
\begin{equation}\label{GoverG}
 \frac{G(t,s)}{G(s,s)}=\begin{cases}
     \displaystyle\frac{\zeta+\frac{\eta}{\alpha}(1-t^\alpha)}{\zeta+\frac{\eta}{\alpha}(1-s^\alpha)} &: s\le t, \\
      & \\
     \displaystyle\frac{\delta + \frac{\gamma}{\alpha}t^\alpha}{\delta + \frac{\gamma}{\alpha}s^\alpha} &: t\le s;
  \end{cases} 
\end{equation}
this expression yields both inequalities in \eqref{Gbds2} for $g_1$ as in \eqref{g2}.

Next, let
\[ u(t,s) = \frac{1}{d}\left[\delta+\frac{\gamma}{\alpha} t^\alpha\right]\left[\zeta+\frac{\eta}{\alpha}\left(1-s^\alpha\right)\right], \]
so that
\[ G(t,s) =  \begin{cases}
     u(s,t) &: s\le t, \\
     u(t,s) &: t\le s.
  \end{cases} \]
Let $r$ be given by \eqref{rgendef}. Then we have
\begin{eqnarray*}
 \min_{\frac{1}{n}\le t \le 1-\frac{1}{n}}G(t,s) 
 &=&
  \begin{cases}
     u(s,1-1/n) &: s\in[0,1/n], \\
     \min\{u(s,1-1/n), u(1/n,s)\} &: s\in[1/n,1-1/n], \\
     u(1/n,s) &: s\in[1-1/n,1],
  \end{cases}  \\
 &=& \begin{cases}
     u(s,1-1/n) &: s\in[0,r], \\
     u(1/n,s) &: s\in[r,1],
  \end{cases} \\
 &=& \begin{cases}
     \frac{1}{d}\left[\delta+\frac{\gamma}{\alpha} s^\alpha\right]\left[\zeta+\frac{\eta}{\alpha}\left(1-\left(1-1/n\right)^\alpha\right)\right] &: s\in[0,r], \\
     \frac{1}{d}\left[\delta+\frac{\gamma}{\alpha} (1/n)^\alpha\right]\left[\zeta+\frac{\eta}{\alpha}\left(1-s^\alpha\right)\right] &: s\in[r,1],
  \end{cases}
\end{eqnarray*}
where $r$ is given in \eqref{rgendef}. By the monotonicity of $G(t,s)$, we have
\[ \max_{0\le t\le 1} G(t,s) = G(s,s) = \frac{1}{d}\left[\delta+\frac{\gamma}{\alpha} s^\alpha\right]\left[\zeta+\frac{\eta}{\alpha}\left(1-s^\alpha\right)\right], \quad s\in[0,1]. \]
Therefore if we take $g_2$ as in \eqref{ggen}, then $G(t,s)$ satisfies \eqref{Gboundsgen}.
Now since 
\[ \frac{\alpha\zeta+\eta\left(1-\left(1-1/n\right)^\alpha\right)}{\alpha\zeta+\eta\left(1-s^\alpha\right)} \ge \frac{\alpha\zeta+\eta\left(1-\left(1-1/n\right)^\alpha\right)}{\alpha\zeta+\eta}, \quad s\in[0,r] \]
and
\[ \frac{\alpha\delta+\gamma(1/n)^\alpha}{\alpha\delta+\gamma  s^\alpha} \ge \frac{\alpha\delta+\gamma(1/n)^\alpha}{\alpha\delta+\gamma}, \quad s\in[r,1], \]
we could in \eqref{Gboundsgen} use the constant $g_3$ given in \eqref{g3def} instead of \eqref{ggen}. This constant is well defined and strictly positive, since $d>0$ in \eqref{gf2} implies neither $\delta$ and $\gamma$ nor $\eta$ and $\zeta$ can simultaneously be zero. 
\end{proof}


\begin{remark}
For conjugate boundary conditions, $\gamma=\eta=1$ and $\delta=\zeta=0$; if $\alpha=1$ and $n=4$, then $g_3\equiv 1/4$, the constant used in \cite[(3.4)]{ht}. For right-focal boundary conditions, $\gamma=\zeta=1$ and $\delta=\eta=0$, so that clearly $g_3\equiv (1/n)^{\alpha}$ for all $\alpha\in(0,1]$ and all integers $n\ge 3$. In the conjugate case specifically, the constant bound \eqref{g3def} is new for fractional derivatives, as the standard Riemann-Liouville fractional derivative does not allow one to calculate a single constant bound; see \cite[Remark 2.2]{bai}.
\end{remark}


 \noindent The following corollary is needed in Section 4 for the main existence theorem and example found there.
      

\begin{corollary} 
Let $\alpha,\beta \in (0,1]$. For every $s \in [0,1]$ we have
$$\max_{t \in [0,1]} G(t,s) \leq \left(\frac{1}{1-\left(\frac34\right)^{\alpha}}\right) \min_{t \in \left[\frac14,\frac34\right]} G(t,s), $$
where  $G(t,s)$ is the Green's function \eqref{gfconj} for the homogeneous problem 
\[ -D^{\beta}D^{\alpha} x(t)=0 \] satisfying the conjugate boundary conditions $x(0)=x(1)=0$.
\end{corollary} 

\begin{proof}
By \eqref{gfconj}, 
\[ \max_{t\in[0,1]}G(t,s) = G(s,s). \]
Then
\begin{eqnarray*}
 \min_{t\in\left[\frac{1}{4},\frac{3}{4}\right]} \frac{G(t,s)}{G(s,s)} 
 & = & \begin{cases} \frac{1-t^\alpha}{1-s^\alpha} &: 0\le s\le t\le 3/4 \\ \frac{t^\alpha}{s^\alpha} &: 1/4\le t\le s\le 1 \end{cases} \\
 &\ge& \begin{cases} \frac{1-(3/4)^\alpha}{1-(0)^\alpha} &: 0 \le s\le 3/4 \\ \frac{(1/4)^\alpha}{(1)^\alpha} &: 1/4\le s\le 1 \end{cases} \\
 & = & 1-(3/4)^\alpha,
\end{eqnarray*}
since $(1/4)^\alpha \ge 1-(3/4)^\alpha$ for all $\alpha\in(0,1]$. One could also use \eqref{Gbdconst} and \eqref{g3def}.
\end{proof}

\section{Fixed Point Preliminaries}

\noindent In this section we will state the fixed point theorem and the definitions that are used in the fixed point theorem which will be used to verify the existence of a positive solution to the fractional-order boundary value problem with conjugate boundary conditions.

\begin{definition}Let $E$ be a real Banach space.  A nonempty closed
convex set $P \subset E$ is called a \emph{cone} if it satisfies the
following two conditions:
\begin{enumerate}
  \item[$(i)$] $x \in P, \lambda \geq 0$ implies $\lambda x \in P$;
  \item[$(ii)$] $x \in P, -x \in P$ implies $x = 0$.
\end{enumerate}
\end{definition}

\noindent Every cone $P \subset E$ induces an ordering in $E$ given by
$$x \leq y  \; \; \mbox{if and only if} \; \;  y - x \in P.$$

\begin{definition} An operator is called completely continuous if it is  
continuous and maps bounded sets into precompact sets.
\end{definition}

\begin{definition}A map $\xi$ is said to be a nonnegative continuous  
concave functional on a cone $P$ of a real Banach space $E$ if  
$$\xi : P \rightarrow [0,\infty)$$
is continuous and
$$\xi(tx + (1-t)y) \geq t\xi(x) + (1-t)\xi(y)$$
for all $x,y \in P$ and $t \in [0,1]$.  Similarly we say the map  
$\phi$ is a nonnegative continuous convex functional on a cone $P$ of a
real Banach space $E$ if  
$$\phi : P \rightarrow [0,\infty)$$
is continuous and
$$\phi(tx + (1-t)y) \leq t\phi(x) + (1-t)\phi(y)$$
for all $x,y \in P$ and $t \in [0,1]$.  We say the map $\psi$ is a sub-linear functional if
$$\psi(tx)\leq t \psi(x) \;\; \hbox{for all} \;\; x \in P, t \in [0,1].$$
\end{definition}


\begin{definition}\label{A1} Let $P$ be a cone in a real Banach space $E$ and  
$\Omega$ be a bounded open subset of $E$ with $0 \in \Omega$.  Then a continuous functional $\phi : P \rightarrow [0,\infty)$ is said to satisfy property $(A1)$ if one of the following conditions hold:
\begin{enumerate}
 \item[$(i)$] $\phi$ is convex, $\phi(0) = 0$, $\phi(x) \neq 0$ if $x \neq 0$, and $\displaystyle{\inf_{x \in P \cap \partial \Omega} \phi(x) > 0}$,
 \item[$(ii)$] $\phi$ is sublinear, $\phi(0) = 0$, $\phi(x) \neq 0$ if $x \neq 0$, and $\displaystyle{\inf_{x \in P \cap \partial \Omega} \phi(x) > 0}$,
 \item[$(iii)$] $\phi$ is concave and unbounded.
\end{enumerate}
\end{definition}


\begin{definition}\label{A2}  Let $P$ be a cone in a real Banach space $E$ and  
$\Omega$ be a bounded open subset of $E$ with $0 \in \Omega$.  Then a continuous functional $\phi : P \rightarrow [0,\infty)$ is said to satisfy property $(A2)$ if one of the following conditions hold:
\begin{enumerate}
 \item[$(i)$] $\phi$ is convex, $\phi(0) = 0$ and $\phi(x) \neq 0$ if $x \neq 0$,
 \item[$(ii)$] $\phi$ is sublinear, $\phi(0) = 0$ and $\phi(x) \neq 0$ if $x \neq 0$, 
 \item[$(iii)$] $\phi(x+y) \geq \phi(x) + \phi(y)$ for all $x,y \in P$, $\phi(0) = 0$, $\phi(x) \neq 0$ if $x \neq 0$.
\end{enumerate}
\end{definition}

\noindent The following theorem is Avery, Henderson, and O'Regan's functional compression-expansion fixed point theorem \cite{aho}, which generalized the functional compression fixed point theorems of Anderson-Avery \cite{aa} and Sun-Zhang \cite{sunfpt}.


\begin{theorem}\label{fpt} Let $\Omega_1$ and $\Omega_2$ be two bounded open sets in a Banach Space $E$ such that $0 \in \Omega_1$ and $\overline{\Omega_1} \subseteq \Omega_2$ and $P$ is a cone in $E$.  Suppose $A:P \cap (\overline{\Omega_2} - \Omega_1) \rightarrow P$ is completely continuous, 
 $\xi$ and $\psi$ are nonnegative continuous  functionals on $P$, and one of the two conditions:
\begin{enumerate}
\item[$(K1)$] $\xi$ satisfies property $(A1)$ with $\xi(Ax) \geq \xi(x)$, for all $x\in P\cap \partial\Omega_1$, 
 and 
$\psi$ satisfies property $(A2)$ with $\psi(Ax) \leq \psi(x)$, for all $x\in P\cap \partial\Omega_2$ , or
\item[$(K2)$] $\xi$ satisfies property $(A2)$ with $\xi(Ax) \leq \xi(x)$, for all $x\in P\cap \partial\Omega_1$, 
and 
$\psi$ satisfies property $(A1)$ with $\psi(Ax) \geq \psi(x)$, for all $x\in P\cap \partial\Omega_2$, \end{enumerate}
\noindent  is satisfied.  Then $A$ has at least one fixed point in $P \cap (\overline{\Omega_2} - \Omega_1)$.
\end{theorem}

\section{Existence of a Positive Solution}

  \noindent Let the Banach space 
  $$E = C[0,1]$$ 

  \noindent be endowed with the maximum norm, 
  $$\left\|x\right\|=\displaystyle{\max_{0\leq t\leq 1}}\left|x(t)\right|,$$ 

  \noindent and define the cone $P \subset E$ by 
    $$P = \left\{ x \in E \left| 
  \begin{array}{l}  
     \mbox{nonnegative valued on [0,1]}, \; \mbox{and}  \\ 
       \|x\| \leq \left(\frac{1}{1-\left(\frac34\right)^{\alpha}}\right)\displaystyle{\min_{t \in \left[\frac14,\frac34\right]} x(t)} 
  \end{array} \right. \right\}.$$

  \noindent Let the nonnegative continuous functionals $\phi$ and $\psi$ be defined on the cone $P$ by 

  \begin{align} 
    & \psi (x) = \min_{t \in \left[\frac14,\frac34\right]} x(t) \;\;\; \mbox{and} 
     \\  
    & \phi(x) = \max_{t \in  [0,1]}x(t) = \|x\|. 
  \end{align} 

The following theorem is our main result.

\begin{theorem} \label{app}
  Let $\alpha,\beta \in (0,1]$ and suppose there exists positive numbers $r$ and $R$ such 
  that 
    $0 < \left(\frac{1}{1-\left(\frac34\right)^{\alpha}}\right) r < R, $ 
  and suppose $f$ satisfies the following conditions: 
  \begin{enumerate} 
    \item[$(i)$]  $f(s,x)\le R(\alpha+\beta)(2\alpha+\beta)$ for all $s\in[0,1]$ and all $x \in [0,R]$, 
    \item[$(ii)$]  $f(s,x)\ge r N$ for all $s\in[1/4,3/4]$ and for all $x \in \left[r,\frac{r}{1-\left(\frac34\right)^{\alpha}}\right]$,
  \end{enumerate} 
where 
\[ N=\left(\left(1-\left(\frac34\right)^\alpha\right)\int_{1/4}^{3/4}G(s,s)d_{\beta}s\right)^{-1} \]
and 
\begin{eqnarray*}
 \left(1-\left(\frac34\right)^\alpha\right)\int_{1/4}^{3/4}G(s,s)d_{\beta}s 
 & = &  \left(1-\left(\frac34\right)^\alpha\right)\left\{\left(\frac34\right)^{\alpha+\beta} \left[\frac{1}{\alpha+\beta}-\frac{(3/4)^\alpha}{2\alpha+\beta}\right]\right. \\ 
 &   & \left.-\left(\frac14\right)^{\alpha+\beta} \left[\frac{1}{\alpha+\beta}-\frac{(1/4)^\alpha}{2\alpha+\beta}\right]\right\}
\end{eqnarray*}
  \noindent Then, the second order conjugate boundary value problem  has at least one positive solution $x^*$  such 
  that 
  $$r \leq \min_{t \in   \left[\frac14,\frac34\right]}x^*(t)  \;\; 
  \mbox{and} 
  \;\; \max_{t \in \left[0,1\right] } x^*(t) \leq R.$$  
\end{theorem}

\begin{proof}

   Define the completely continuous operator $A$ by 
      $$Ax(t)=\int_{0}^{1}G(t,s)f(s,x(s))d_{\beta}s$$ 
  then if we can show that $A$ has a fixed point in $P$ then we have verified the existence of a positive solution.
  Let $x \in P$, then from properties of $G(t,s)$ we have that $Ax(t) \geq 0$ 
  and 
\begin{eqnarray*} 
    \phi(Ax)  & = & \max_{t \in [0,1]}\int_{0}^{1}G(t,s)f(s,x(s))d_{\beta}s\\
&\leq& 
    \left(\frac{1}{1-\left(\frac34\right)^{\alpha}}\right) \min_{t \in \left[\frac14,\frac34\right]} \int_{0}^{1}G(t,s)f(s,x(s))d_{\beta}s\\ 
    & =& \left(\frac{1}{1-\left(\frac34\right)^{\alpha}}\right) \psi(Ax)
  \end{eqnarray*} 
thus, $Ax \in P$ and we have verified that $A: P \rightarrow P$. 

  \vspace{.15in}

  \noindent For all $x \in P$  we have 
  $$\psi(x) \leq \phi(x),$$ 
  thus if we let
$$\Omega_1=\{x \; : \; \psi(x) < r\} \;\; \mbox{and}  \;\; \Omega_2=\{x \; : \; \phi(x) < R\}$$
we have that 
$$0 \in \Omega_1 \;\; \mbox{and}  \;\;\overline{\Omega_1} \subseteq \Omega_2$$

since if $x \in \overline{\Omega_1}$ then
$$\min_{t \in   \left[\frac14,\frac34\right]} x(t) \leq r$$
hence since $x \in P$ we have
$$\max_{t \in   \left[0,1\right]} x(t) \leq \left(\frac{1}{1-\left(\frac34\right)^{\alpha}}\right)\min_{t \in   \left[\frac14,\frac34\right]} x(t) \leq  \left(\frac{1}{1-\left(\frac34\right)^{\alpha}}\right) r < R.$$
Clearly $\Omega_1$ and $\Omega_2$ being bounded open subsets of $P$.

  \vspace{.1in} 

  \noindent {\bf Claim 1:}  If $x \in P \cap \partial \Omega_2$, then $\phi(Ax) \leq \phi(x)$.     
  \vspace{.1in} 
\noindent Let $x \in \partial \Omega_2$, thus $\phi(x) = R$ hence by condition $(i)$ we have

  \begin{eqnarray*} 
    \phi (Ax)  & = & \max_{t \in \left[0,1\right] }\int_{0}^{1}G(t,s)f(s,x(s))d_{\beta}s\\
&\leq &  R(\alpha+\beta)(2\alpha+\beta)\int_{0}^{1}G(s,s)d_{\beta}s \\ 
    & = & R = \phi(x). 
  \end{eqnarray*}

  \vspace{.1in}

  \noindent {\bf Claim 2:}  If $x \in P \cap \partial \Omega_1$, 
  then $\psi(Ax) \geq \psi(x)$.     
  \vspace{.1in} 
\noindent Let $x \in \partial \Omega_1$, thus $\psi(x) = r$ and $\|x\| \leq \frac{r}{1-\left(\frac34\right)^{\alpha}}$, hence by condition $(ii)$ we have

\begin{eqnarray*} 
  \psi (Ax) & = & \min_{t \in   \left[\frac14,\frac34\right]} \int_{0}^{1}G(t,s)f(s,x(s))d_{\beta}s\\
            &\ge& rN  \left(1-\left(\frac34\right)^\alpha\right)\int_{1/4}^{3/4} G(s,s)d_{\beta}s \\ 
            & = & r = \psi(x). 
\end{eqnarray*}

  \vspace{.1in} 

\noindent Clearly $\phi$ satisfies property $(A1)(i)$ and $\psi$ satisfies property $(A2)(iii)$ thus the hypothesis $(K1)$ of Theorem \ref{fpt} is satisfied, and therefore $A$ has a fixed point in $\overline{\Omega_2}-\Omega_1$.

\end{proof}

\begin{example} 
To compare our results with those in \cite[Example 3.1]{bai}, where the authors use the Riemann-Liouville fractional derivative of order $\frac32\in(1,2]$, we take $\alpha=1$, $\beta=\frac12$, $r=\frac{11}{1000}$, $R=\frac{9}{25}$, and $f(s,x)=1+\frac{1}{4}\sin s+x^2$ in Theorem \ref{app} to get the following.
One can check that $0 < 4r < R$, and $f$ satisfies the following conditions: 
  \begin{enumerate} 
    \item[$(i)$]  $f(s,x)\le \frac{15}{4}R=\frac{27}{20}$ for all $s\in[0,1]$ and all $x \in \left[0,\frac{9}{25}\right]$, 
    \item[$(ii)$]  $f(s,x)\ge \frac{960r}{33\sqrt{3}-17}=\frac{264}{25(33\sqrt{3}-17)}$ for all $s\in\left[\frac14,\frac34\right]$ and for all $x \in \left[r,4r\right]$.
  \end{enumerate} 
Thus by Theorem \ref{app} the $\frac32$-order conjugate boundary value problem  
  \[ -D^{0.5}x'(t)=1+\frac{1}{4}\sin t+x(t)^2, \quad x(0)=x(1)=0 \]
  has at least one positive solution $x^*$ such that 
  $$ \frac{11}{1000} \leq \min_{t \in   \left[\frac14,\frac34\right]}x^*(t)  \;\; 
  \mbox{and} 
  \;\; \max_{t \in \left[0,1\right] } x^*(t) \leq \frac{9}{25}. $$
In \cite[Example 3.1]{bai}, the result is the existence of a positive solution $x^\dagger$ such that
\[ \frac{1}{14} \le \max_{t \in \left[0,1\right] } x^\dagger(t) \le 1, \]
with no information on the minimum value of the function.
\end{example}



\begin{thebibliography}{99}

\bibitem{aa} D.R. Anderson and R.I. Avery, Fixed point theorem of cone expansion and compression of functional type, {\em J. Difference Equations Appl.}{\bf 8} (2002), pp. 1073--1083.

\bibitem{aho} R.I. Avery, J. Henderson and D. O'Regan, Functional compression-expansion fixed point theorem, {\em Electron. J. Differential Equations} {\bf 2008} (2008), No. 22, pp.
1-12. 

\bibitem{bai} Zhanbing Bai and Haishen L\"{u},  Positive solutions for boundary value problem of nonlinear fractional differential equation,
{\em J. Math. Anal. Appl.} {\bf 311} Issue 2 (2005) 495--505.

\bibitem{ahmad} Bashir Ahmad and J. J. Nieto, Existence results for a coupled system of nonlinear fractional differential
equations with three-point boundary conditions, {\em Computers Math. Appl}, {\bf 58} (2009) 1838--1843.

\bibitem{cabada} Alberto Cabada and Guotao Wang, Positive solutions of nonlinear fractional differential equations with integral boundary value conditions,
{\em J. Math. Anal. Appl.} {\bf 389} Issue 1 (2012) 403--411.

\bibitem{chai} Guoqing Chai and Songlin Hu, Existence of positive solutions for a fractional high-order three-point boundary value problem,
{\em Advances in Difference Equations} {\bf 2014} 2014:90.

\bibitem{chenli} Yanli Chen and Yongxiang Li, The existence of positive solutions for boundary value problem of nonlinear fractional differential equations, 
{\em Abst. Appl. Anal.} {\bf 2014} (2014), Article ID 681513, 7 pages.

\bibitem{chenliu} Shengping Chen and Yuji Liu, Solvability of boundary value problems for fractional order elastic beam equations,
{\em Advances in Difference Equations} {\bf 2014} 2014:204.

\bibitem{eloe} P. W. Eloe and J. T. Neugebauer, Existence and comparison of smallest eigenvalues for a fractional boundary-value problem, {\em Electron. J. Differential Equations} {\bf 2014} (2014), No. 43, pp. 1--10. 

\bibitem{udita} U. Katugampola, A new fractional derivative with classical properties, \emph{J. American Math. Soc.}, arXiv:1410.6535v2.

\bibitem{hammad} M. Abu Hammad and R. Khalil, Abel's formula and Wronskian for conformable fractional differential equations, \emph{International J. Differential Equations Appl.}, {\bf 13} No. 3 (2014) 177--183.

\bibitem{han} Xiaoling Han and Hongliang Gao, Existence of positive solutions for eigenvalue problem of nonlinear fractional differential equations,
{\em Advances in Difference Equations} {\bf 2012} 2012:66.

\bibitem{ht} J. Henderson and H. B. Thompson, Multiple symmetric positive solutions for a second order boundary value problem, \emph{Proc. Amer. Math. Soc.}, {\bf 128} Number 8 (2000) 2373--2379.

\bibitem{khalil} R. Khalil, M. Al Horani, A. Yousef, and M. Sababheh, A new definition of fractional derivative, {\em J. Computational Appl. Math.}, {\bf 264} (2014) 65--70.

\bibitem{liu} Xiaoyou Liu and Yiliang Liu, Fractional differential equations with fractional non-separated boundary conditions, {\em Electron. J. Differential Equations} {\bf 2013} (2013), No. 25, pp. 1--13.

\bibitem{mahm} N. I. Mahmudov and S. Unul, Existence of solutions of $\alpha\in(2,3]$ order fractional three-point boundary value problems with integral conditions, {\em Abst. Appl. Anal.} {\bf 2014} (2014), Article ID 198632, 12 pages.

\bibitem{sunfpt} J. Sun and G. Zhang, A generalization of the cone expansion and compression fixed point theorem and applications, {\sl Nonlinear Anal.} {\bf67} (2007), 579-586.

\bibitem{sun} Yongping Sun and Xiaoping Zhang, Existence and nonexistence of positive solutions for fractional-order two-point boundary value problems,
{\em Advances in Difference Equations} {\bf 2014} 2014:53.

\bibitem{wu}  J. Wu and X. Zhang, Eigenvalue problem of nonlinear semipositone higher order fractional differential equations, {\em Abst. Appl. Anal.} {\bf 2012} (2012), Art. ID 740760, 14 pp.

\bibitem{zhang}  X. Zhang, L. Liu, B. Wiwatanapataphee, and Y. Wu, Positive solutions of eigenvalue problems for a class of fractional differential equations with derivatives, {\em Abst. Appl. Anal.} {\bf 2012} (2012), Art. ID 512127, 16 pp.

\bibitem{zhou}  Xiangbing Zhou and Wenquan Wu, Uniqueness and asymptotic behavior of positive solutions for a fractional-order integral boundary-value problem, {\em Electron. J. Differential Equations} {\bf 2013} (2013), No. 37, pp. 1--10. 

\end{thebibliography}
\end{document}